\documentclass[11pt]{amsart}

\usepackage{amssymb}
\usepackage{amsmath}
\usepackage{amsthm}
\usepackage{graphicx,psfrag}
\usepackage{enumerate}

\usepackage[dvips]{hyperref}
\hypersetup{pdfborder={}}

\theoremstyle{plain}
\newtheorem{theorem}{Theorem}[section]
\newtheorem{lemma}[theorem]{Lemma}
\newtheorem{corollary}[theorem]{Corollary}

\theoremstyle{definition}
\newtheorem{definition}[theorem]{Definition}
\newtheorem{question}[theorem]{Question}
\newtheorem{conjecture}[theorem]{Conjecture}

\theoremstyle{remark}
\newtheorem{remark}[theorem]{Remark}
\newtheorem*{claim}{Claim}
\newtheorem*{remark*}{Remark}

\newtheorem*{acknowledgments}{Acknowledgments}
\newtheorem*{notation}{Notation}

\numberwithin{figure}{section}

\newcommand{\Int}{\mathrm{int}}

\begin{document}

\title[On the degeneration of tunnel numbers]{On the degeneration of tunnel numbers under connected sum}
\author{Tao Li}
\address{Department of Mathematics \\
 Boston College \\
 Chestnut Hill, MA 02467\\
U.S.A.}
\email{taoli@bc.edu}
\thanks{The first author is partially supported by NSF grants DMS-1005556 DMS-1305613, and the second author is partially supported by NSFC 11171108.}

\author{Ruifeng Qiu}
\address{Department of Mathematics \\
 East China Normal University \\
 Shanghai, 200241\\ P.R.China}
\email{rfqiu@math.ecnu.edu.cn}

\begin{abstract}
We show that, for any integer $n\ge 3$, there is a prime knot $k$ such that (1) $k$ is not meridionally primitive, and (2) for every $m$-bridge knot $k'$ with $m\leq n$, the tunnel numbers satisfy $t(k\# k')\le t(k)$.  This gives counterexamples to a conjecture of Morimoto and Moriah on tunnel number under connected sum and meridionally primitive knots.
\end{abstract}

\maketitle

\section{Introduction}\label{Sintro}

Let $M$ be a compact 3--manifold. If there is a closed surface $S$ that cuts $M$ into two compression bodies $V$ and $W$ such that $S=\partial_{+} V=\partial_{+} W$, then we say that $S$ is a Heegaard surface of $M$, and the decomposition, denoted by $M=V\cup_{S} W$, is called a Heegaard splitting of $M$. If $g(S)$ is minimal among all Heegaard surfaces of $M$, then $V\cup_{S} W$ is called a minimal Heegaard splitting, and the Heegaard genus $g(M)$ is defined to be $g(S)$.  Any orientable compact 3--manifold admits a Heegaard splitting.

Let $k$ be a knot in the 3--sphere $S^{3}$ and let $E(k)=S^3-N(k)$, where $N(k)$ is an open tubular neighborhood of $k$. There is always a collection of disjoint and embedded arcs $\tau_{1},\ldots, \tau_{t}$ in $S^3$, such that $\tau_i\cap k=\partial\tau_i$ for each $i$ and
$H=S^3-N((\cup_{i=1}^{t}\tau_{i})\cup k)$ is a handlebody.
This means that $\partial H$ is a Heegaard surface of $E(k)$.  These arcs $\tau_{1},\ldots, \tau_{t}$ are called unknotting tunnels of $k$ and we say that they form a tunnel system for $k$.  The tunnel number of $k$, denoted by $t(k)$, is the minimal number of arcs in a tunnel system for $k$.  Let $g(E(k))$ be the Heegaard genus of the knot exterior $E(k)$.  Clearly $g(E(k))=t(k)+1$. 

For two knots $k_{1}$ and $k_{2}$, we denote the connected sum of $k_{1}$ and $k_{2}$ by $k_{1}\# k_{2}$.  Given any tunnel systems for $k_1$ and $k_2$, one can obtain a tunnel system for $k_1\# k_2$ by putting together the tunnel systems for $k_1$ and $k_2$ plus an extra tunnel lying in the decomposing 2--sphere.  This means that the tunnel numbers of these knots satisfy the following inequality:
$$t(k_1\# k_2)\le t(k_1)+t(k_2)+1.$$
An interesting question in knot theory and 3--manifold topology is to study the relation between the tunnel number of a composite knot and the tunnel numbers of its factors.  The following is a list of some results in this direction:

(1) Norwood proved that tunnel number one knots are prime \cite{N}, Scharlemann and Schultens proved that the tunnel number of the connected sum of $n$ nontrivial knots is at least $n$ and $1/3$ of the sum of the tunnel numbers of the factors. See \cite{SS1} and \cite{SS2}.

(2) For two knots $k_{1}$ and $k_{2}$, the equality $t(k_{1}\# k_{2})=t(k_{1})+t(k_{2})+1$ is called the super additivity of tunnel number under connected sum.  Morimoto \cite{Mo1} first showed that the super additivity does not always hold.  Kobayashi \cite{K} gave examples of composite knots with $t(k_{1}\# k_{2})\le t(k_{1})+t(k_{2})-n$ for any integer $n$.  Nogueira \cite{No} found the first prime knots $k_1$, $k_2$ with $t(k_{1}\# k_{2})=t(k_{1})+t(k_{2})-2$. 
Schultens \cite{Sch1} proved that if both $k_{1}$ and $k_{2}$ are small knots, then $t(k_{1}\# k_{2})\geq t(k_{1})+t(k_{2})$, and Morimoto \cite{Mo2} proved that if both $k_{1}$ and $k_{2}$ are small knots, then the super additivity holds if and only if none of $k_{1}$ and $k_{2}$ is meridionally primitive (see section~\ref{SA} for definition).  As an extension of Morimoto's result, Gao, Guo and Qiu \cite{GQ} proved that if a minimal Heegaard splitting of $E(k_{1})$ has high distance while a minimal Heegaard splitting of $E(k_{2})$ has distance at least 3, then $t(k_{1}\# k_{2})=t(k_{1})+t(k_{2})+1$.

Note that if a knot $k_1$ is meridionally primitive, then $t(k_{1}\# k_{2})< t(k_{1})+t(k_{2})+1$ for any knot $k_2$ (see section~\ref{SA} for details).  Based on the result in \cite{Mo2}, Morimoto conjectured that $t(k_{1}\# k_{2})=t(k_{1})+t(k_{2})+1$ if and only if  none of $k_{1}$ and $k_{2}$ is meridionally primitive (this was also conjectured by Moriah \cite{Mor1}).  Kobayashi and Rieck \cite{KR} gave a counterexample to this conjecture, but none of the two factors $k_{1}$ and $k_{2}$ in their example is prime.  The following is a modified version of Morimoto's conjecture (this is also conjectured by Moriah, see \cite[Conjecture 7.14]{Mori2}).

\begin{conjecture}\label{conj1}
 Suppose that both $k_{1}$ and $k_{2}$ are prime. Then $t(k_{1}\# k_{2})=t(k_{1})+t(k_{2})+1$ if and only if  neither $k_{1}$ nor $k_{2}$ is meridionally primitive.
\end{conjecture}

The main result of this paper is the following:

\begin{theorem}\label{Tmain}
For any integer $n\geq 3$, there is a prime knot $k$ such that
\begin{enumerate}
  \item  $k$ is not meridionally primitive, and
  \item  for every $m$-bridge knot $k'$ with $m\leq n$, the tunnel numbers satisfy $t(k\# k')\le t(k)$, in other words, the Heegaard genera of the knot exteriors satisfy $g(E(k\# k'))\le g(E(k))$.
\end{enumerate}
\end{theorem}

By choosing the knot $k'$ in Theorem~\ref{Tmain} to be prime and not meridionally primitive, we have a counterexample to Conjecture~\ref{conj1}. Furthermore, we have the following immediate corollaries.

\begin{corollary}\label{CMori} For any collection of knots $k_{1},\ldots, k_{m}$, there is a prime knot $k$ such that $k$ is not meridionally primitive and $t(k_{i}\# k)\le t(k)$ for all $i$.
\end{corollary}

\begin{corollary}\label{CMori} For any integer $m$,
there exist knots $k_1$ and $k_2$ that are prime and not meridionally primitive, such that $t(k_1\# k_2)\le t(k_1)+t(k_2)-m$.
\end{corollary}

In the proof of the main theorem, we only give an upper bound on $t(k\# k')$. However, it is conceivable that the tunnel number of the connected sum should be at least as large as the tunnel number of any of its factors.  This is an interesting question in its own right.

\begin{question}\label{Q}
Are there two knots $k$ and $k'$ in $S^3$ with $t(k\# k')<t(k)$?
\end{question}

If there is a degree-one map $f\colon M\to N$ between two closed orientable 3--manifolds $M$ and $N$, a difficult question in 3--manifold topology is whether or not $g(M)\ge g(N)$ always holds.  Note that there is a degree-one map from $E(k\# k')$ to $E(k)$.  Thus a positive answer to Question~\ref{Q} also gives an example of 3--manifolds $M$ and $N$ such that there exists a degree-one map $f\colon M\to N$ but $g(M)<g(N)$.

\begin{acknowledgments}
The first author would like to thank Yoav Moriah for many helpful conversations.
\end{acknowledgments}

\section{Annulus sum and meridionally primitive knots}\label{SA}

\begin{notation}
Throughout this paper, for any subspace $X$ in a manifold, $N(X)$ denotes an open regular neighborhood of $X$, $\Int(X)$ denotes the interior of $X$ and $\overline{X}$ denotes the closure of $X$.  For any knot $k$ in $S^3$, we use $E(k)$ to denote the knot exterior $S^3-N(k)$.  For any compact 3--manifold $X$, we use $g(X)$ to denote its Heegaard genus.
\end{notation}

\begin{definition} Let $E(k)$ be a knot exterior and
 $E(k)=V\cup_{S} W$ a Heegaard splitting of $E(k)$ with $\partial_{-} W=\partial E(k)$.  So $V$ is a handlebody.  Let $r$ be a slope in the boundary torus $\partial E(k)$. If there are an essential disk $B$ in $V$ and a properly embedded annulus $A$ in $W$ with $\partial A=\alpha\cup\beta$ such that
\begin{enumerate}
  \item $\alpha\subset\partial E(k)$ is a curve of slope $r$, and
  \item $\beta\subset\partial_+W$ and $\beta\cap\partial B$ is a single point,
\end{enumerate}
then we say that the Heegaard splitting is $r$-primitive.  Note that, since $\beta\cap\partial B$ is a single point, $\beta$ is a primitive curve in the handlebody $V$.
\end{definition}

\begin{definition} Let $k$ be a knot in $S^{3}$ and $E(k)$ the exterior of $k$. We denote by $\mu$ the meridional slope in $\partial E(k)$.  We say that $k$ is meridionally primitive or $\mu$-primitive if $E(k)$ has a minimal genus Heegaard splitting that is $\mu$-primitive.  Note that this is equivalent to saying that $k$ admits a $(t(k),1)$ position, see \cite{Mo2, KR}.
\end{definition}

Given two 3--manifolds $M$ and $N$ with nonempty boundary, an annulus sum of $M$ and $N$ is a 3--manifold obtained by identifying an annulus $A_M\subset \partial M$ and an annulus $A_N\subset\partial N$.  Note that, for any two knots $k_1$ and $k_2$ in $S^3$, $E(k_1\# k_2)$ is an annulus sum of $E(k_1)$ and $E(k_2)$, identifying a pair of meridional annuli in $\partial E(k_1)$ and $\partial E(k_2)$ (in this paper, an annulus is meridional if its core curve has meridional slope in the boundary torus).

Let $X=M\cup_A N$ be an annulus sum of two compact 3--manifolds $M$ and $N$.  One can obtain a Heegaard surface of $X$ by connecting Heegaard surfaces of $M$ and $N$ using a tube through the gluing annulus $A$.  Thus $g(X)\le g(M)+g(N)$.
Suppose $M=E(k)$ and a minimal Heegaard splitting of $M$ is $r$-primitive, where $r$ is the slope of the core curve of the annulus $A$.  Then it is easy to see that the Heegaard splitting of $X=M\cup_A N$ constructed this way (using the $r$-primitive Heegaard splitting of $M$) is stabilized and hence $g(X)< g(M)+g(N)$.
Thus a meridionally primitive Heegaard splitting of $E(k)$ provides a clear picture how the Heegaard genus of the connected sum of two knots can degenerate.  We summarize this well-known fact as the following lemma, see \cite{Mo2} for a proof.

\begin{lemma}\label{Lprim}
Let $k$ be a nontrivial knot in $S^3$ and $N$ any compact irreducible 3--manifold with boundary. Let $A$ be an annulus in $\partial N$ and let $M=E(k)\cup_A N$ be an annulus sum identifying $A$ to a meridional annulus in $\partial E(k)$.  If $k$ is $\mu$-primitive, then $g(M)<g(E(k))+g(N)$.
\end{lemma}

The following is an immediate corollary of Lemma~\ref{Lprim}, and we will use it later to prove that a knot is not $\mu$-primitive.

\begin{corollary}\label{Cprim}
Let $M$ and $N$ be as in Lemma~\ref{Lprim}.  Suppose $N=T^2\times I$.  If $g(M)=g(E(k))+1$, then $k$ is not $\mu$-primitive.
\end{corollary}

Before we proceed, we would like to clarify some terminology that we will use in constructing surfaces:
Given a surface $S$ and an embedded arc $\gamma$ with $\partial\gamma=\gamma\cap S$, we can construct a new surface $S'$ by first removing two disk neighborhoods of the two endpoints of $\gamma$ in $S$ and then connecting the resulting boundary circles by an annulus (or tube) along $\gamma$.  We say that the surface $S'$ is obtained by adding a tube to $S$ along the arc $\gamma$.

\section{The construction}\label{Sc}

Let $n\ge 3$ be any integer.
The knot $k$ in Theorem~\ref{Tmain} is constructed by gluing together two $n$-string tangles.  We require that each tangle exterior has a high-distance Heegaard splitting.  The following is a construction of such tangles.

Let $n$ be the integer above.  We fix a large number $d>0$ ($d$ is assumed to be sufficiently large relative to $n$). By \cite{MMS}, for any number $d$, there is a knot $K$ in $S^3$ with tunnel number $t\ge n$ such that $E(K)$ has a Heegaard splitting of distance at least $d$.  Let $\tau_1,\dots, \tau_t$ be the unknotting tunnels of $K$. We may assume that $\tau_1,\dots, \tau_t$ are disjoint and view $G=(\cup_{i=1}^t\tau_i)\cup K$ as a trivalent graph in $S^3$.  Since $t\ge n$, we may find a tree $Y$ in the graph $G$ that contains exactly $n-1$ tunnels and is disjoint from the remaining $t-n+1$ tunnels.  Let $N(Y)$ be a small open neighborhood of $Y$. We require that each tunnel $\tau_i$ is either totally in $Y$ or totally outside $N(Y)$.  Since $Y$ is a tree, $N(Y)$ is an open 3--ball.  Let $B=S^3-N(Y)$.  So $B$ is a 3--ball and $K\cap B$ is an $n$-string tangle in $B$.

Let $N(G)$ be a neighborhood of $G$ in $S^3$ that contains $N(Y)$. Let $H=S^3-N(G)$ and $S=\partial H$.  By our construction of $G$, $H$ is a handlebody and $S$ is a Heegaard surface of $E(K)$.  Moreover, the distance of the Heegaard splitting along $S$ is at least $d$, where $d$ is a number that can be assumed to be arbitrarily large. Since $N(Y)\subset N(G)$, $H\subset B$. Let $M_B=B-N(K\cap B)$ be the tangle exterior, where $N(K\cap B)$ is a small open neighborhood of $K\cap B$ in $B-H$.  By our construction, $M_B-\Int(H)$ can be obtained from the compression body $E(K)-\Int(H)$ by removing a small neighborhood of part of its spine.  Thus $M_B-\Int(H)$ is also a compression body and $S$ is a Heegaard surface of the tangle exterior $M_B$.  Moreover, the disk complex of the compression body $M_B-\Int(H)$ is a nontrivial subcomplex of the disk complex of the compression body $E(K)-\Int(H)$.  Thus the distance of the Heegaard splitting of $M_B$ along $S$ is at least $d$.  In particular, by assuming $d\ge 2$, we may assume that the Heegaard splitting of $M_B$ (along $S$) is strongly irreducible, and hence $\partial M_B$ is incompressible in $M_B$ \cite{CG}.

We take two copies of $B$, denoted by $B_1$ and $B_2$.  Then we glue $B_1$ to $B_2$ along the boundary sphere in such a way that the union of the two tangles is a knot.  This is our knot $k$.  We will show next that $k$ is a prime knot and it is not $\mu$-primitive.  These follow from that the tangle exterior has a high-distance Heegaard splitting.

Let $S_0=\partial B_1=\partial B_2$ be the 2--sphere, $\Gamma=k\cup S_0$, and $M_i=B_i-N(\Gamma)$ ($i=1,2$). So $M_i$ is basically the exterior of the tangle in $B_i$.  Let $F_i=\partial M_i$ and let $M_0=\overline{N(\Gamma)}-N(k)$, where $N(k)$ is a small tubular neighborhood of $k$ in $N(\Gamma)$.  Thus $F_i=M_0\cap M_i$, and the knot exterior $E(k)$ can be obtained by gluing $M_1$ and $M_2$ to $M_0$ along $F_1$ and $F_2$ respectively. So we may view $E(k)=M_1\cup M_0\cup M_2$ and view $M_0$, $M_1$ and $M_2$ as submanifolds of $E(k)$.  By our earlier conclusion, $F_i$ is incompressible in $M_i$.  Moreover, each tangle exterior $M_i$ has a Heegaard surface $S_i$ such that the distance of the Heegaard splitting of $M_i$ along $S_i$ is at least $d$, where $d$ is a number that we can arbitrarily choose.

Next we consider $M_0$.  Note that $\partial M_0$ has 3 components: $F_1$, $F_2$ and the torus $T=\partial E(k)$. By our construction, $M_0$ can be obtained by gluing a product $\Sigma\times I$ to a product $T\times I$, where $I=[0,1]$, $\Sigma$ is a $(2n)$-hole sphere and $T$ is a torus, and the gluing map identifies each annulus in $(\partial\Sigma)\times I$ to an annulus in $T\times\{1\}$.  Thus $T\times\{0\}=T=\partial E(k)$, and each $F_i$ can be obtained from a component of $\Sigma\times\partial I$ by adding tubes/annuli (from $T\times\{0\}$) connecting its boundary circles.  We view $T\times I$ and $\Sigma\times I$ as submanifolds of $M_0$.

Let $A_1,\dots, A_{2n}$ be the $2n$ components of $(\partial\Sigma)\times I$.  So $\cup_{i=1}^{2n}A_i$  divides $M_0$ into two submanifolds $T\times I$ and $\Sigma\times I$.  It is clear from our construction that each $A_i$ is an incompressible and $\partial$-incompressible annulus properly embedded in $M_0$.  This also implies that $F_1$ and $F_2$ are incompressible in $M_0$.  Thus $F_1$ and $F_2$ are incompressible in $E(k)$.

\begin{lemma}\label{Lprime}
If $d>2$, the knot $k$ is prime.
\end{lemma}
\begin{proof}
Suppose $k$ is not prime. Then $E(k)$ contains an essential annulus $Q$ with meridional $\partial$-slope.  Since $F_1$ and $F_2$ are incompressible in $E(k)$, after isotopy, we may assume that if $Q\cap F_i\ne\emptyset$, $Q\cap M_i$ is a collection of essential annuli in $M_i$. Since the Heegaard distance of the splitting in $M_i$ is at least $d>2$, by \cite[Theorem 3.1]{S}, $M_i$ contains no essential annulus.  Thus $Q\cap F_i=\emptyset$ for both $i$ and hence $Q\subset M_0$.  The following claim is well known. \vskip 3mm

\begin{claim} Let $F$ be an incompressible surface in $\Sigma\times I$ with $\partial F\subset(\partial\Sigma)\times I$, then either $F$ is $\partial$-parallel to a subannulus of a component $A_i$ of $(\partial\Sigma)\times I$, or $F$ is in the form $\Sigma\times \bigl\{x\bigr\}$, where $x\in I$.
\end{claim}

Let $(\partial\Sigma)\times I=\cup_{i=1}^{2n} A_{i}$ be as above. Since each $A_i$ is an essential annulus in $M_0$, after isotopy, each component of $Q\cap(\Sigma\times I)$ is an incompressible annulus in $\Sigma\times I$. Since $n>1$, by the claim, each component of $Q\cap(\Sigma\times I)$ must be an annulus $\partial$-parallel to a subannulus of $A_{i}$ for some $i$.   This means that after isotopy, $Q\subset T\times I$ with $\partial Q\subset T\times\{0\}$.  Clearly $T\times I$ contains no such essential annulus, a contradiction.
\end{proof}

Next we give two constructions of a genus-$(2n)$ Heegaard surface for $M_0$.  These constructions imply that the Heegaard genus $g(M_0)\le 2n$.  The first construction gives a Heegaard surface of $M_0$ that separates $F_1$ from $F_2$.  In the second construction, $F_1$ and $F_2$ lie on the same side of the Heegaard surface.

Construction 1: We start with a peripheral surface $F_1'$ in $M_0$ parallel to $F_1$, which has genus $n$.  We describe a neighborhood of $T\times I$ in $M_0$ using a (2-dimensional) schematic picture Figure~\ref{FHeeg}(a).  Note that $X$ in Figure~\ref{FHeeg}(a) is a picture of an annulus and $X\times S^1$ is our neighborhood of $T\times I$ in $M_0$, Figure~\ref{FHeeg} is a picture with $n=3$ and the shaded regions in Figure~\ref{FHeeg} denote parts of $\Sigma\times I$ that are neighborhoods of the annuli $A_i$'s.  Now we add $n$ tubes to $F_1'$ along $n$ unknotted arcs which are parallel to arcs in $T\times\{1\}$ and circularly connect the $n$ annuli of $F_1'\cap (T\times I)$, see the dashed arcs in Figure~\ref{FHeeg}(b) for a picture of these $n$ arcs (this is slightly misleading since all other arcs in Figure~\ref{FHeeg} denote surfaces but the dashed arcs denote arcs connecting the surfaces).  Since $g(F_1)=n$, the resulting surface $S_F$ has genus $2n$.  Moreover, if one maximally compresses $S_F$ outside the added tubes, one ends up with a surface parallel to $F_2$ and a torus parallel to $T\times\{0\}$.  Thus $S_F$ is a Heegaard surface for $M_0$ of genus $2n$.

Construction 2:  Let $F_i'$ ($i=1,2$) be a peripheral surface in $M_0$ parallel to $F_i$.  We may assume $F_i'\cap (\Sigma\times I)$ is of the form $\Sigma\times\{x\}$ ($x\in I$).  There is a subarc of an $I$-fiber of $\Sigma\times I$ connecting $F_1'$ to $F_2'$, and we add a tube along such a vertical arc to $F_1'\cup F_2'$.  Since $g(F_1)=g(F_2)=n$, the resulting surface has genus $2n$.  If we compress this surface along the meridional curve of the tube, we get $F_1'\cup F_2'$, and if we maximally compress it on the other side, we get a peripheral torus parallel to $T$.  This means that the resulting surface is also a Heegaard surface for $M_0$ of genus $2n$.

\begin{figure}
  \centering
\psfrag{a}{(a)}
\psfrag{b}{(b)}
\psfrag{c}{$X\times S^1$ is the 3-dimensional picture}
\psfrag{X}{$X$}
\psfrag{1}{$F_1'$}
  \includegraphics[width=4in]{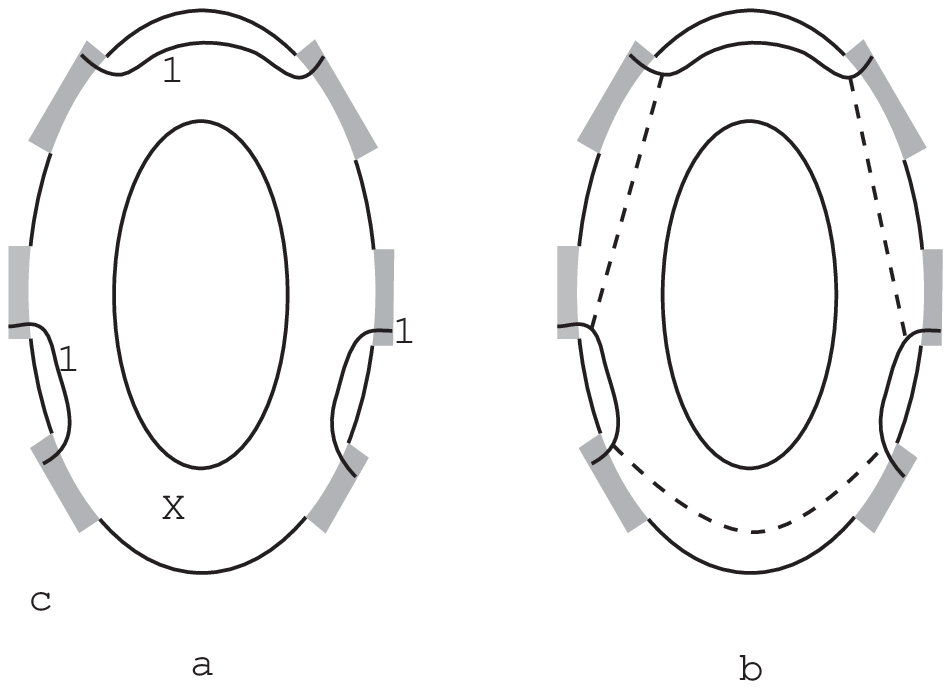}
  \caption{}
  \label{FHeeg}
\end{figure}

Let $A_0$ be a meridional annulus in $\partial E(k)=T\times\{0\}\subset\partial M_0$.  Let $M_0^+$ be the annulus sum of $M_0$ and a product $\hat{T}=T^2\times I$, identifying $A_0$ to an essential annulus in $\partial \hat{T}$.  We view $A_0$ as an annulus properly embedded in $M_0^+$, dividing $M_0^+$ into $M_0$ and $\hat{T}$.  Thus $\partial M_0^+$ has 4 components: $F_1$, $F_2$, a torus $T_0$ which is the union of $T\times\{0\}-A_0$ and an annulus in $\partial \hat{T}$, and a torus component $T_1$ from $\partial \hat{T}$.  Since $g(M_0)\le 2n$, by section~\ref{SA}, $g(M_0^+)\le 2n+1$.

Our next goal is to compute the Heegaard genera of $M_0$ and $M_0^+$.  As $M_0^+=M_0\cup \hat{T}$, we may view $\Sigma\times I$ and $T\times I$ as submanifolds of both $M_0$ and $M_0^+$.

\begin{lemma}\label{LSigma}
\begin{enumerate}
  \item Either $g(M_0)=2n$ or $M_0$ contains a minimal genus Heegaard surface $S$ such that $S\cap(\Sigma\times I)$ is a connected surface in the form $\Sigma\times\{x\}$ ($x\in I$).
  \item Either $g(M_0^+)=2n+1$ or $M_0^+$ contains a minimal genus Heegaard surface $S$ such that $S\cap(\Sigma\times I)$ is a connected surface in the form $\Sigma\times\{x\}$ ($x\in I$).
\end{enumerate}
\end{lemma}

\begin{proof}
We prove the two parts of the lemma at the same time, since $M_0$ and $M_0^+$ are similar.
Let $S$ be a minimal genus Heegaard surface in our 3--manifold ($M_0$ or $M_0^+$).  We will show that we can either isotope $S$ or change $S$ to a possibly different Heegaard surface satisfying the conditions in the lemma.  We have two cases.

\vspace{10pt}

\noindent \emph{Case (a)}. $S$ is strongly irreducible.

\vspace{10pt}

Since $(\partial\Sigma)\times I$ consists of essential annuli in our 3--manifold, by \cite[Lemma 3.7]{L5} (also see \cite{L4}), after isotopy, we may assume that at most one component of $S\cap(\Sigma\times I)$ is strongly irreducible and all other components are incompressible in $\Sigma\times I$ (recall that a surface is strongly irreducible if it is compressible on both sides and every compressing disk on one side meets every compressing disk on the other side).  Moreover, after pushing components of $S\cap(\Sigma\times I)$ out of $\Sigma\times I$, we may assume that no component of $S\cap(\Sigma\times I)$ lies in a collar neighborhood of the annuli $(\partial\Sigma)\times I$.  So by the claim in the proof of Lemma 3.1, each incompressible component of $S\cap(\Sigma\times I)$ is of the form $\Sigma\times\{x\}$ ($x\in I$).

\begin{claim} There is a minimal genus Heegaard surface $S'$ of the 3--manifold ($M_0$ or $M_0^+$) such that each component of $S'\cap(\Sigma\times I)$ is of the form $\Sigma\times\{x\}$ ($x\in I$).
\end{claim}
\begin{proof}[Proof of the Claim]
By our assumption above, either $S$ satisfies the claim or exactly one component of $S\cap(\Sigma\times I)$ is strongly irreducible.
Suppose $P$ is a strongly irreducible component of $S\cap(\Sigma\times I)$.  Next we are basically repeating an argument in \cite[Lemma 3.9 and Lemma 3.10]{L5}.  We call the two sides of $P$ plus and minus sides.  Let $P_+$ and $P_-$ be the surfaces obtained by maximally compressing $P$ in $\Sigma\times I$ on the plus and minus sides respectively and deleting all the 2--sphere components. We may assume that $\partial P_+=\partial P=\partial P_{-}=P\cap P_+=P\cap P_-$.  There is a connected submanifold of $\Sigma\times I$ between $P$ and $P_+$ (similar to a compression body), which we denote by $W_{+}$, and similarly, a connected submanifold between $P$ and $P_-$, denoted by $W_{-}$.   Furthermore, $W_{+}\cap W_{-}=P$. Hence there is a connected region $W=W_+\cup W_-$ between $P_+$ and $P_-$ which contains $P$.  Clearly $\partial W=P_+\cup P_-$.
Since $P$ is strongly irreducible, by \cite{CG} (also see \cite[Lemma 5.5]{S}), $P_\pm$ is incompressible in $\Sigma\times I$.  Hence each component of $P_\pm$ is either a $\partial$-parallel annulus or a surface of the form $\Sigma\times\{x\}$.

Let $\Gamma$ be a component of $P_{+}$ or $P_{-}$. We define the shadow of $\Gamma$, denoted by $P(\Gamma)$, as follows.
If $\Gamma$ is a $\partial$-parallel annulus, then $P(\Gamma)$ is the solid torus bounded by $\Gamma$ and a subannulus of $A_{i}\subset(\partial\Sigma)\times I$. By our assumption that $P$ is not in a collar neighborhood of $(\partial\Sigma)\times I$, $W$ must lie outside $P(\Gamma)$.
If $\Gamma$ is in the form $\Sigma\times\{x\}$, then $\Gamma$ divides $\Sigma\times I$ into two interval-bundles, one of which contains $W$, and we define $P(\Gamma)$ to be the interval-bundle that does not contain $W$.   Thus, in either case, $W$ lies outside $P(\Gamma)$.  Moreover, since $\partial W=P_+\cup P_-$, $W$ and $P(\Gamma)$ lie on the two sides of $\Gamma$ and $W\cap P(\Gamma)=\Gamma$.

Let $\Gamma_1$ and $\Gamma_2$ be two components of $P_+ \amalg P_-$.  We say that $\Gamma_1$ and $\Gamma_2$ are non-nested if $\Int(P(\Gamma_{1}))\cap \Int(P(\Gamma_2))= \emptyset$, and we claim that $\Gamma_1$ and $\Gamma_2$ are always non-nested.  By our definition of $P(\Gamma)$, if $\Int(P(\Gamma_{1}))\cap \Int(P(\Gamma_2))\neq \emptyset$, then either $P(\Gamma_1)\subset P(\Gamma_2)$ or $P(\Gamma_2)\subset P(\Gamma_1)$.  Suppose the claim is false and $P(\Gamma_2)\subset P(\Gamma_1)$.  Since $W\cap P(\Gamma_2)=\Gamma_2$ and since $\partial W=P_+\cup P_-$, this means that $W\subset P(\Gamma_1)$, which contradicts that $W$ lies outside $P(\Gamma_1)$.

Since $W\cap P(\Gamma)=\Gamma$ for each component $\Gamma$ of $P_+\amalg P_-$, and since $\partial W=P_+\cup P_-$, the union of $W$ and all the shadows $P(\Gamma)$'s is the whole of $\Sigma\times I$. Thus $W$ is isotopic to $\Sigma\times I$ and the isotopy can be realized by pushing each $\Gamma$ to $\partial P(\Gamma)-\Gamma$ in $P(\Gamma)$. This implies that $P_+\amalg P_-$ has exactly two components in the form $\Sigma\times\{x\}$ and all other components are non-nested $\partial$-parallel annuli.

Without loss of generality, we may assume that $P_+$ has at least one component in the form $\Sigma\times\{x\}$, and  $P_+$ has $k$ components.
As shown in Figure~\ref{Fmove}(a, b), we can connect the $k$ components of $P_+$ by adding $k-1$ tubes to $P_+$ along $k-1$ unknotted arcs which can be isotoped into $(\partial\Sigma)\times I$, and the resulting connected surface, which we denote by $P'$, is a connected sum of all the components of $P_+$.  $P'$ and $P$ have some similar properties: If one compresses $P'$ along the meridional curves of the $k-1$ tubes, one gets back $P_+$; if one maximally compresses $P'$ on the other side, one gets a collection of non-nested $\partial$-parallel surfaces in $\Sigma\times I$, which means that the resulting surface is $P_-$.  Moreover, by our construction, $P$ can be obtained by adding 1--handles (or tubes) to $P_+$ on the same side of $P_+$.  Since $P_+$ has $k$ components and $P$ is connected, to obtain $P$, we need to add at least $k-1$ tubes to $P_+$.  This implies that $\chi(P)\le\chi(P')$.   Now we can replace $P$ by $P'$.  The union of $P'$ and all the other components of $S\cap(\Sigma\times I)$ and $S-(\Sigma\times I)$ is a closed surface, which we denote by $S'$.  Since $S$ is a Heegaard surface, the properties of $P'$ and $P_\pm$ above imply that $S'$ is also a Heegaard surface (see the proof of Lemma 3.9 in [7]). Since $\chi(P)\le\chi(P')$, we have $\chi(S)\le\chi(S')$ and $g(S)\ge g(S')$.  Since $S$ is a minimal genus Heegaard surface, $S'$ must also be a minimal genus Heegaard surface and  $g(S')=g(S)$.

\begin{figure}
  \centering
\psfrag{a}{(a)}
\psfrag{b}{(b)}
\psfrag{c}{(c)}
\psfrag{A}{$(\partial\Sigma)\times I$}
\psfrag{S}{$P_+$}
\psfrag{2}{$P'$}
  \includegraphics[width=5in]{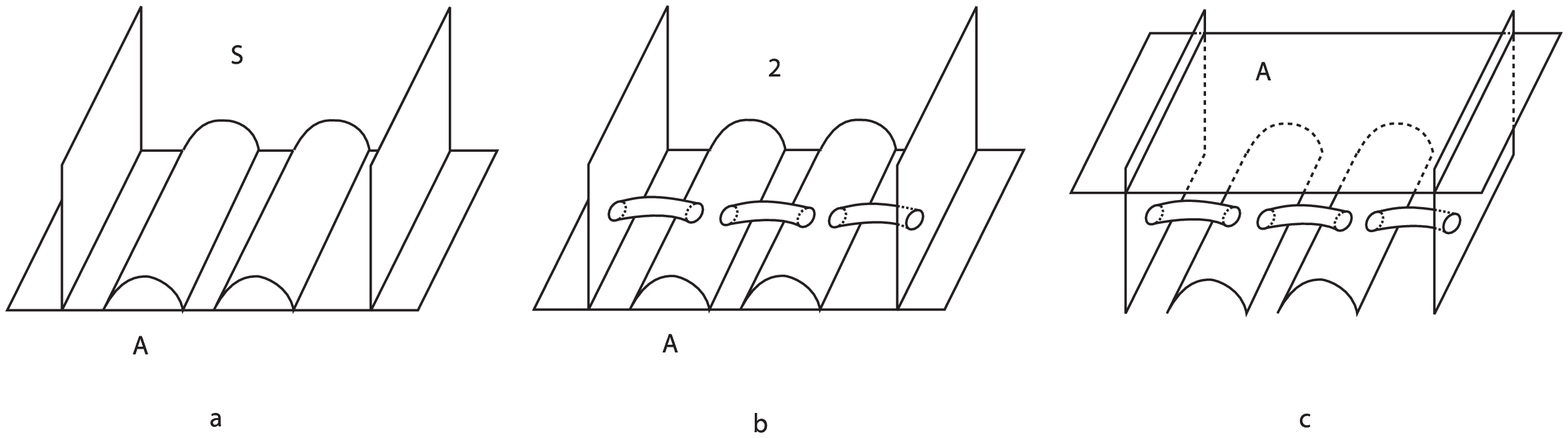}
  \caption{}
  \label{Fmove}
\end{figure}

As shown in Figure~\ref{Fmove}(c), we can isotope $S'$ by pushing the $\partial$-parallel annuli in $P_+$ and all the added tubes across $(\partial\Sigma)\times I$ and out of $\Sigma\times I$.  After this isotopy, each component of $S'\cap (\Sigma\times I)$ is of the form $\Sigma\times\{x\}$.
\end{proof}

Next we consider the number of components in $S'\cap(\Sigma\times I)$.
If $S'\cap(\Sigma\times I)$ has at least 3 components, then $\chi(S')\le 3\chi(\Sigma)=3(2-2n)=6-6n=2-2(3n-2)$ and $g(S')\ge 3n-2$.  Since $n\ge 3$, we have $g(S')\ge 2n+1$.

Recall that our 3--manifold is either $M_0$ or $M_0^+$, and by our construction prior to the lemma, we have $g(M_0)\le 2n$ and $g(M_0^+)\le 2n+1$.

We now separately discuss $M_0$ and $M_0^+$.
Suppose our 3--manifold is $M_0$. Since $S'$ is a minimal genus Heegaard surface and since $g(M_0)\le 2n$, the conclusion above implies that $S'\cap(\Sigma\times I)$ has at most 2 components.  If $S'\cap(\Sigma\times I)$ has exactly two components, then the two components of $\Sigma\times\partial I$ lie on the same side of the Heegaard surface $S'$, and hence $F_1$ and $F_2$ are boundary components of the same compression body in the Heegaard splitting of $M_0$ along $S'$.  As $g(F_1)=g(F_2)=n$, this implies that $g(S')\ge g(F_1)+g(F_2)=2n$.  As $S'$ is a minimal genus Heegaard splitting and since $g(M_0)\le 2n$, this means that $g(S')=g(M_0)=2n$.  Thus either  $g(M_0)=2n$ or $S'\cap (\Sigma\times I)$ has a single component.  Hence part (1) of the lemma holds in Case (a).

Suppose our 3--manifold is $M_0^+$. Since $S'$ is a minimal genus Heegaard surface, we have $g(S')=g(M_0^+)\le 2n+1$. Let $M_0^+=W\cup_{S'} V$ be the Heegaard splitting along $S'$.  We may assume $S'$ is strongly irreducible, otherwise we are in Case (b) below.  By the conclusion above, if $S'\cap(\Sigma\times I)$ has at least 3 components, then $g(S')\ge 2n+1$, hence $g(M_0^+)=2n+1$ and part (2) of the lemma holds.  To prove part (2), it remains to consider the possibility that $S'\cap(\Sigma\times I)$ has 2 components.  Suppose this is the case, then similar to the argument above, both $F_1$ and $F_2$ lie in the same compression body, say $W$, in the Heegaard splitting $M_0^+=W\cup_{S'} V$ along $S'$.  As $\partial_-W$ contains both $F_1$ and $F_2$, we have $g(S')\ge g(\partial_-W)\ge g(F_1)+g(F_2)=2n$.


Note that two components of $\partial M_0^+$ are tori.  If $\partial_- W$ also contains a torus component of $\partial M_0^+$, then $g(S')\ge g(\partial_-W)\ge g(F_1)+g(F_2)+1=2n+1$, which implies that $g(M_0^+)=2n+1$, and hence part (2) of the lemma holds.  Thus it remains to consider the case that both tori in $\partial M_0^+$ lie in the compression body $V$.  Suppose the lemma is false in this case and $g(S')\le 2n$.  By our conclusion earlier that $g(S')\ge g(\partial_-W)\ge g(F_1)+g(F_2)=2n$, we have $g(S')=2n$ and hence the compression body $W$ can be obtained by adding a single 2--handle to (a product neighborhood of) $S'$ along a separating curve in $S'$.  Thus $M_0^+$ can be obtained by adding a single 2--handle to $V$.  Since the core curve of the 2--handle is a separating (and null-homologous) curve in $S'$, we have the identity $H_1(M_0^+)=H_1(V)$ on the first homology.

The compression body $V$ can be obtained by adding 1--handles to a product neighborhood of $\partial_-V$.  Since $\partial_-V$ contains both boundary tori $T_0$ and $T_1$ of $\partial M_0^+$, we may view the homology groups $H_1(T_0)$ and $H_1(T_1)$ as direct summands of $H_1(V)$.  By our conclusion $H_1(M_0^+)=H_1(V)$ above,  $H_1(T_0)$ and $H_1(T_0)$ are direct summands of $H_1(M_0^+)$.  Recall that $M_0^+$ is an annulus sum of $M_0$ and the product $\hat{T}=T^2\times I$.  Thus there is a vertical annulus $\alpha\times I\subset\hat{T}=T^2\times I$ that is also a properly embedded annulus in $M_0^+$ connecting $T_0$ to $T_1$.  This means that an essential closed curve in $T_0$ is homologous (in $M_0^+$) to a curve in $T_1$, which contradicts our conclusion that $H_1(T_0)$ and $H_1(T_1)$ are direct summands of $H_1(M_0^+)$.

\vspace{10pt}

\noindent \emph{Case (b)}. $S$ is weakly reducible.

\vspace{10pt}


We consider the untelescoping of the Heegaard splitting, see \cite{ST1}.  Since our 3--manifold is a submanifold of $S^3$, every closed surface in our 3--manifold is separating.  So by \cite{ST1}, there are a collection of separating closed incompressible surfaces $Q_1,\dots, Q_m$ that divide our 3--manifold ($M_0$ or $M_0^+$) into $m+1$ submanifolds $N_0,\dots, N_m$.   There is a strongly irreducible Heegaard surface $P_i$ in each $N_i$ such that $S$ is an amalgamation of these $P_i$'s along the incompressible surfaces $Q_i$'s.  Since $S$ is assumed to be weakly reducible, $g(P_i)\le g(S)-1$ and $g(Q_i)\le g(S)-1$ for each $i$. Since $S$ is a minimal Heegaard surface, $Q_{i}$ is not $\partial$-parallel in the 3--manifold ($M_{0}$ or $M_{0}^{+}$), and $Q_{i}$ is not parallel to $Q_{j}$ if $i\neq j$.

Recall that $\Sigma\times I$ is viewed as a submanifold and $(\partial\Sigma)\times I$ is a collection of essential annuli in our 3--manifold ($M_{0}$ or $M_{0}^{+}$).  By the claim in the proof of Lemma~\ref{Lprime}, after isotopy, we may assume that each component of $Q_i\cap(\Sigma\times I)$ is of the form $\Sigma\times\{x\}$.  Thus each component of $N_i\cap (\Sigma\times I)$ is a product $\Sigma\times J$ ($J\subset I$).  Since each $P_i$ is strongly irreducible, similar to Case (a), after isotopy, we may assume that
\begin{enumerate}
  \item each component of $Q_i\cap(\Sigma\times I)$ is of the form $\Sigma\times\{x\}$,
  \item at most one component of $P_i\cap(\Sigma\times I)$ is strongly irreducible and all other components are incompressible in $\Sigma\times I$,
  \item no component of $P_i\cap(\Sigma\times I)$ lies in a collar neighborhood of the annuli $(\partial\Sigma)\times I$.
\end{enumerate}

Let $\Sigma\times J$ ($J\subset I$) be a component of $N_i\cap(\Sigma\times I)$ for some $i$.  We may assume that each component of $(\partial\Sigma)\times J$ is an essential annulus in $N_{i}$.  Now we apply the argument in the proof of the Claim in Case (a) on $\Sigma\times J$ and $P_i$.  As in the Claim in Case (a), after isotopy and replacing $P_i$ by a possibly different Heegaard surface $P_i'$, we may assume that each component of $P_i'\cap(\Sigma\times I)$ is of the form $\Sigma\times\{x\}$.  Note that since $S$ is a minimal genus Heegaard surface, $P_i$ and $P_i'$ must be minimal genus Heegaard surfaces of $N_i$.

If $P_i'\cap(\Sigma\times J)$ has more than two components, then $\chi(P_i')\le 3(2-2n)$.  Similar to Case (a), since $n\ge 3$, this means that $g(P_i)=g(P_i')\ge 2n+1$.  However, since $g(P_i)\le g(S)-1$, we have $g(S)\ge 2n+2$.  This contradicts that $g(M_0)\le 2n$ and $g(M_0^+)\le 2n+1$.  Thus $P_i'\cap(\Sigma\times J)$ has at most two components.

Suppose $P_i'\cap(\Sigma\times J)$ has two components.  Then since $P_i'$ is a Heegaard surface of $N_i$, both components of $\Sigma\times\partial J$ lie on the same side of $P_i'$.  Since $P_i'$ is not parallel to a component of $\partial N_i$, we have $\chi(P_i')\le\chi(\Sigma\times\partial J)-2=2(2-2n)-2=2-4n$.  This means that $g(P_i')\ge 2n$.  Since $g(P_i')=g(P_i)\le g(S)-1$, we have $g(S)\ge 2n+1$.
If our 3--manifold is $M_0$, then $g(M_0)=g(S)\le 2n$ and this is a contradiction.  If our 3--manifold is $M_0^+$, then $g(M_0^+)=g(S)\le 2n+1$.  So we have $g(M_0^+)=2n+1$ and part (2) of the lemma holds.  Thus, to finish the proof of the lemma, we may assume that $P_i'\cap(\Sigma\times J)$ has exactly one component for each component $\Sigma\times J$ of $N_i\cap(\Sigma\times I)$ and for any $i$.

Note that, by amalgamating these Heegaard surfaces $P_i'$ along the incompressible surfaces $Q_i$, we can obtain a minimal genus Heegaard surface $S'$ for our 3--manifold ($M_0$ or $M_0^+$).  Since $P_i'$ intersects each component of $N_i\cap(\Sigma\times I)$ in a connected surface of the form $\Sigma\times\{x\}$, after the amalgamation, $S'\cap(\Sigma\times I)$ is a connected surface of the form $\Sigma\times\{x\}$ (see \cite[Figure 3]{Sch} for a picture of amalgamation).  We can also see this through the untelescoping, which is a rearrangement of 1-- and 2-- handles, see \cite{ST1}.  Since $P_i'\cap(\Sigma\times J)$ is a connected surface of the form $\Sigma\times\{x\}$ for every component of $N_i\cap(\Sigma\times I)$ and for each $i$, if we maximally compress $P_i'$ on either side, we can choose the compressions to be disjoint from $\Sigma\times I$.  Thus we may view that all the 1-- and 2--handles (in the untelescoping for $S'$) are outside $\Sigma\times I$.  Hence $S'\cap(\Sigma\times I)$ is a connected surface of the form $\Sigma\times\{x\}$ and the lemma holds in Case (b).
\end{proof}

\begin{lemma}\label{Lgenus}
 $g(M_0)=2n$ and $g(M_0^+)=2n+1$.
\end{lemma}
\begin{proof}
The idea of the proof is to first replace $\Sigma\times I$ by $annuli\times I$, which changes the 3--manifold ($M_0$ or $M_0^+$) into a product of a planar surface and $S^1$.  Then we apply the classification of Heegaard surfaces in such products \cite{Sch} to get a lower bound on the Heegaard genus.

Similar to the proof of Lemma~\ref{LSigma}, we consider the two cases $M_0$ and $M_0^+$ at the same time.  Let $S$ be a minimal genus Heegaard surface in our 3--manifold ($M_0$ or $M_0^+$).

By Lemma~\ref{LSigma}, we may assume that $S\cap(\Sigma\times I)$ is a connected surface of the form $\Sigma\times\{x\}$.  This implies that
\begin{enumerate}
  \item  $F_1$ and $F_2$ lie on different sides of $S$, and
  \item  if we maximally compress $S$ on either side, all the compressions can be chosen to be disjoint from $\Sigma\times I$.
\end{enumerate}

Recall that $F_1$ is the union of a component of $\Sigma\times\partial I$ and $n$ annuli which we denote by $E_1\dots, E_n$.
Let $A_1, A_2,\dots, A_{2n}$ be the $2n$ annuli in $(\partial\Sigma)\times I$.  Without loss of generality, we may suppose $E_i$ connects $A_{2i-1}$ to $A_{2i}$, $i=1,\dots,n$.
Now we construct a new 3--manifold by replacing $\Sigma\times I$ with $n$ copies of $A\times I$ ($A$ is an annulus) so that each copy of $A\times I$ connects a pair of annuli $A_{2i-1}\cup A_{2i}$ (i.e.~$A_{2i-1}\cup A_{2i}=(\partial A)\times I$).  This construction is for both $M_0$ and $M_0^+$, and we use $N$ and $N^+$ to denote the new 3--manifolds constructed from  $M_0$ and $M_0^+$ respectively.

Recall that $\overline{M_0-(\Sigma\times I)}=X\times S^1$ and $\overline{M_0^+-(\Sigma\times I)}=X^+\times S^1$, where $X$ is an annulus and $X^+$ is a pair of pants, see Figure~\ref{Fplanar}(a, b) for pictures of $X$ and $X^+$ (with $n=3$). Figure~\ref{Fplanar}(a) can be compared with Figure~\ref{FHeeg}(a). Note that these annuli $A_j$'s are in $T\times\{1\}\subset T\times I$.  Moreover, since $F_1$ is a boundary component of $M_0$, after renaming these annuli if necessary, we may assume that $A_1, A_2,\dots, A_{2n}$ lie in $T\times\{1\}$ in a cyclic order, see Figure~\ref{Fplanar}(a,b). Thus, as shown in Figure~\ref{Fplanar}(c, d), the new 3--manifolds $N=Y\times S^1$ and $N^+=Y^+\times S^1$, where $Y$ is a planar surface with $n+2$ boundary circles and $Y^+$ is a planar surface with $n+3$ boundary circles.

\begin{figure}
  \centering
\psfrag{a}{(a)}
\psfrag{b}{(b)}
\psfrag{c}{(c)}
\psfrag{d}{(d)}
\psfrag{1}{$A_1$}
\psfrag{2}{$A_2$}
\psfrag{3}{$A_3$}
\psfrag{4}{$A_4$}
\psfrag{5}{$A_5$}
\psfrag{6}{$A_6$}
\psfrag{x}{$X$}
\psfrag{X}{$X^+$}
\psfrag{y}{$Y$}
\psfrag{Y}{$Y^+$}
  \includegraphics[width=4in]{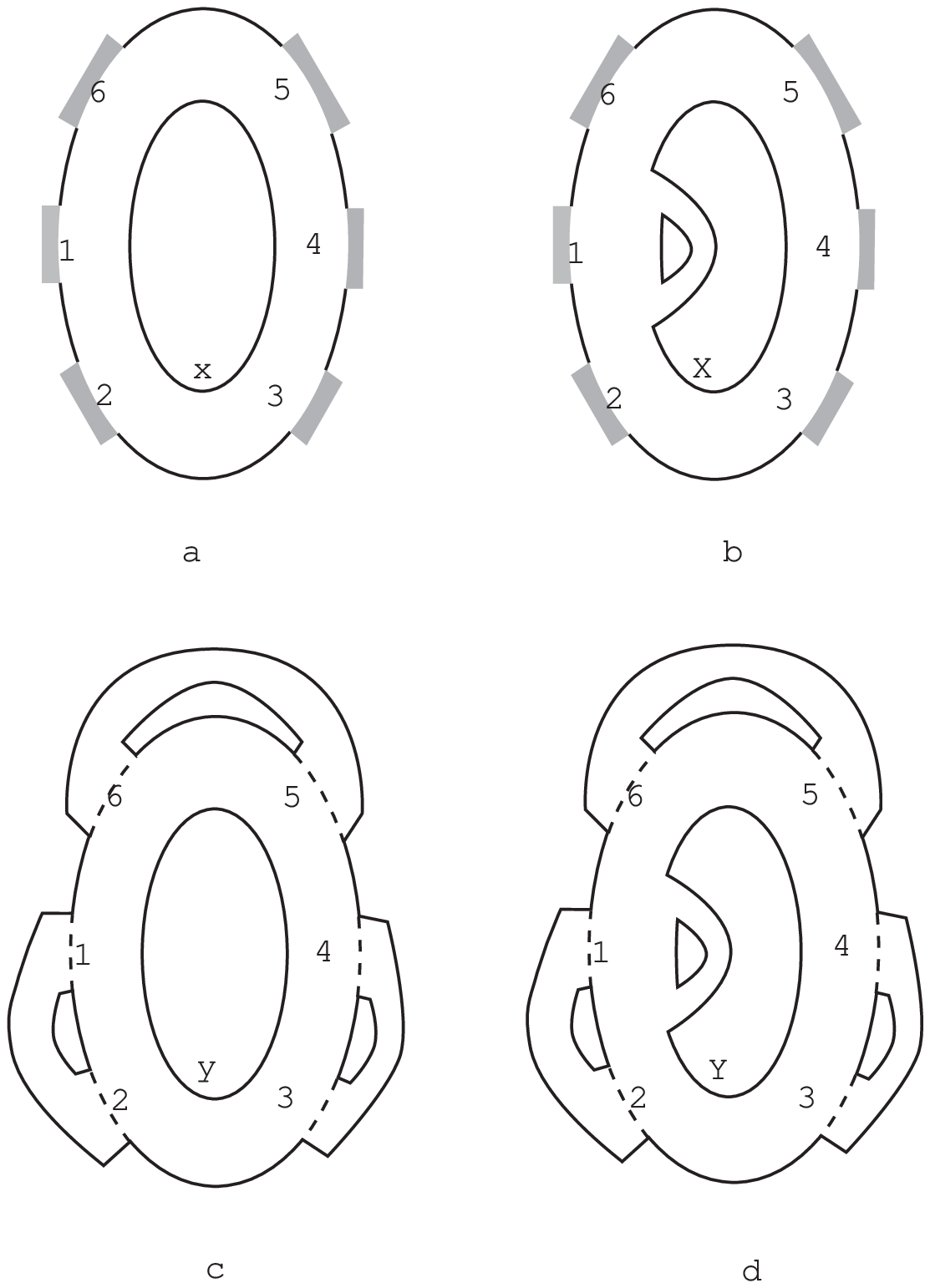}
  \caption{}
  \label{Fplanar}
\end{figure}

Since $S\cap(\Sigma\times I)$ is a connected surface of the form $\Sigma\times\{x\}$, similar to our construction of the manifolds $N$ and $N^+$, we can replace $S\cap(\Sigma\times I)$ by $n$ horizontal annuli of the form $A\times\{x\}$ in the $n$ copies of $A\times I$.  This yields a new closed surface, which we denote by $S_N$, in the new 3--manifold ($N$ or $N^+$).  Clearly $\chi(S_N)=\chi(S)-(2-2n)$.

Recall the $F_1$ and $F_2$ lie on two different sides of $S$ and we can choose maximal compressions for $S$ totally outside $\Sigma\times I$.  Since $S$ is a Heegaard surface, if we maximally compress $\overline{S-(\Sigma\times I)}$ on one side, the resulting surface contains a collection of $n$ non-nested $\partial$-parallel annuli (in $\overline{M_0-(\Sigma\times I)}$ or $\overline{M_0^+-(\Sigma\times I)}$) connecting the $n$ pairs of annuli $A_{2i-1}\cup A_{2i}$, while if we compress it on the other side, the resulting surface has a collection of $n$ non-nested $\partial$-parallel annuli connecting the $n$ pairs of annuli $A_{2i}\cup A_{2i+1}$ ($A_{2n+1}=A_1$).  This implies that the $2n$ circles of $S\cap ((\partial\Sigma)\times I)$ are connected by one component of $\overline{S-(\Sigma\times I)}$ and hence $S-(\Sigma\times I)$ is connected.  Thus $S_N$ is connected.  Moreover, since $S_N$ is obtained from $S$ by replacing the portion in $\Sigma\times I$ with annuli, maximal compressions for $S$ on either side (outside $\Sigma\times I$) correspond to maximal compressions for $S_N$.  Since the $n$ copies of $A\times I$ connect the $n$ pairs of annuli $A_{2i-1}\cup A_{2i}$ (see Figure~\ref{Fplanar}(c,d)), maximally compressing $S_N$ on either side yields peripheral tori (in $N$ or $N^+$).
This implies that $S_N$ is in fact a Heegaard surface in our new manifold ($N$ or $N^+$).  Since $\chi(S_N)=\chi(S)-(2-2n)$, we have $g(S_N)=g(S)-n+1$ and $g(S)=g(S_N)+n-1$.

In our construction, $N=Y\times S^1$ and $N^+=Y^+\times S^1$, where $Y$ is a planar surface with $n+2$ boundary circles and $Y^+$ is a planar surface with $n+3$ boundary circles.  By \cite{Sch}, the Heegaard genus $g(N)=n+1$ and $g(N^+)=n+2$.  Thus, if our manifold is $M_0$, we have $g(M_0)=g(S)=g(S_N)+n-1\ge g(N)+n-1=2n$.  Since we already know $g(M_0)\le 2n$, we have $g(M_0)=2n$.  Similarly, if our manifold is $M_0^+$, we have $g(M_0^+)=g(S)=g(S_N)+n-1\ge g(N^+)+n-1=2n+1$.  Since we already know $g(M_0^+)\le 2n+1$, we have $g(M_0^+)=2n+1$.
\end{proof}

\begin{lemma}\label{Lg}
Let $E=E(k)$ be the exterior of $k$ and let $M_0$, $M_1$ and $M_2$ be as above.
Let $E^+$ be the annulus sum of $E(k)$ and $T^2\times I$ which identifies a meridional annulus in $\partial E$ to a nontrivial annulus in $T^2\times\partial I$.
Suppose $g(M_1)=g_1$ and $g(M_2)=g_2$.  If $d$ is sufficiently large, then $g(E)=g_1+g_2$ and $g(E^+)=g_1+g_2+1$.
\end{lemma}
\begin{proof}
By assuming the Heegaard distance of $M_1$ and $M_2$ to be sufficiently large (compared with $g_1+g_2$), by \cite[Theorem 2]{QWZ}, a minimal genus Heegaard splitting of $E$ (resp.~$E^+$) is an amalgamation of splittings of $M_1$, $M_2$ and $M_0$ (resp.~$M_0^+$) along $F_1$ and $F_2$.  By Lemma~\ref{Lgenus}, $g(M_0)=2n$ and $g(M_0^+)=2n+1$.  Thus
$$g(E)=g(M_1)+g(M_0)+g(M_2)-g(F_1)-g(F_2)=g_1+2n+g_2-n-n=g_1+g_2,$$
$$g(E^+)=g(M_1)+g(M_0^+)+g(M_2)-g(F_1)-g(F_2)=g_1+g_2+1.$$
\end{proof}
\begin{remark}
The Heegaard genus of the tangle exteriors $M_1$ and $M_2$ can be as small as $n+1$.  Thus by Lemma~\ref{Lg}, $g(E(k))$ can be as small as $2n+2$.
\end{remark}

\begin{corollary}\label{Cg}
If $d$ is sufficiently large, then the knot $k$ is not $\mu$-primitive.
\end{corollary}
\begin{proof}
This follows from Lemma~\ref{Lg} and Corollary~\ref{Cprim}.
\end{proof}

\section{Proof of the main theorem}
In this section, we prove Theorem~\ref{Tmain}.  Our main task is to give an upper bound on the Heegaard genus of the knot exterior $E(k'\# k)$.  The idea is that the bridge sphere of $k'$ and the punctured sphere which divides $k$ into two tangles can be put together in $E(k'\# k)$ to produce a Heegaard surface for $E(k'\# k)$.

Theorem~\ref{Tmain} follows from Lemma~\ref{Lmain}, Lemma~\ref{Lprime} and Corollary~\ref{Cg}.

\begin{lemma}\label{Lmain}
For any integer $n\ge 3$, let $k$ be the knot constructed in section~\ref{Sc}. If $d$ is sufficiently large, then for any nontrivial knot $k'$ which admits an $m$-bridge sphere with $m\leq n$, $g(E(k'\# k))\le g(E(k))$.
\end{lemma}
\begin{proof}
Recall that $E(k'\# k)$ is an annulus sum of $E(k')$ and $E(k)$.

Since $k'$ has an $m$-bridge sphere and since $m\leq n$, $k'$ also has a (possibly non-minimal) bridge sphere of index $n$.
We start with an $n$-bridge sphere for $k'$.  This bridge sphere for $k'$ corresponds to a $2n$-hole sphere $Q_1$ properly embedded in $E(k')$, and $\partial Q_1$ consists of $2n$ meridional circles.  The $2n$ circles in $\partial Q_1$ divide the torus $\partial E(k')$ into $2n$ annuli which we denote by $C_1,\dots,C_{2n}$ where $C_i$ is adjacent to $C_{i+1}$ for each $i$ ($C_{2n+1}=C_1$). Since $Q_1$ is a bridge sphere, if we maximally compress $Q_1$ on one side, we get a collection of $n$ $\partial$-parallel annuli parallel to $C_1, C_3,\dots, C_{2n-1}$, and if we maximally compress $Q_1$ on the other side, we get a collection of $n$ $\partial$-parallel annuli parallel to $C_2, C_4,\dots, C_{2n}$.

Recall that $k$ is constructed by gluing together two $n$-string tangles $B_1$ and $B_2$.  The punctured sphere $\partial B_1$ corresponds to a $2n$-hole sphere $Q_2$ properly embedded in $E(k)=M_1\cup M_0\cup M_2$, where $M_0$, $M_1$ and $M_2$ are as in section~\ref{Sc}.  We may assume that $Q_2\subset M_0$ and $Q_2$ is obtained by extending $\Sigma\times\{1/2\}$ to the boundary torus $\partial E(k)$.  The $2n$ circles in $\partial Q_2$ divide the torus $\partial E(k)$ into $2n$ annuli which we denote by $D_1,\dots,D_{2n}$ where $D_i$ is adjacent to $D_{i+1}$ for each $i$ ($D_{2n+1}=D_1$).  By our construction, the union of $Q_2$ and the $n$ annuli $D_1, D_3,\dots, D_{2n-1}$ is a closed surface parallel to a component, say $F_1$, of $\partial M_0$, and the union of $Q_2$ and the $n$ annuli $D_2, D_4,\dots, D_{2n}$ is a closed surface parallel to $F_2$.

Let $\gamma$ be an essential arc of the annulus $D_1$.  We push $\gamma$ slightly into the interior of $E(k)$ and call the resulting arc $\gamma'$ ($\partial\gamma'=\gamma'\cap Q_2$).
 We add a tube to $Q_2$ along $\gamma'$ and call the resulting surface $Q_2'$.  If we compress $Q_2'$ once along the meridian of the tube, we get back $Q_2$, and if we compress $Q_2'$ once on the other side, we get a $\partial$-parallel annulus parallel to $D_1$ and a surface isotopic to $Q_2\cup D_1$.

Let $\hat{C}$ be a meridional annulus in $\partial E(k')$ containing $C_2,\dots, C_{2n}$ with $\partial\hat{C}\subset C_1$, and let $\hat{D}$ be a meridional annulus in $\partial E(k)$ containing $D_2,\dots, D_{2n}$ with $\partial\hat{D}\subset D_1$.  By identifying $\hat{C}$ to $\hat{D}$, we obtain the manifold $E(k'\# k)$.  Moreover, we may assume that each $C_i$ is identified with $D_i$ ($i=2,\dots,2n$), and $Q=Q_1\cup Q_2'$ is a closed orientable surface of genus $2n$ embedded in $E(k'\# k)$.  Next we maximally compress $Q$ on either side and we perform the compressions on $Q_1$ and $Q_2'$ separately in $E(k')$ and $E(k)$ respectively.   By the compression properties of $Q_1$ and $Q_2'$ described earlier, we see that if we maximally compress $Q$ on one side (the $C_2$ and $D_2$ side), we get a surface parallel to $F_2$ (because the union of $Q_2$ and the $n$ annuli $D_2, D_4,\dots, D_{2n}$ is a closed surface parallel to $F_2$), and if we maximally compress $Q$ on the other side, we get two components: one is parallel to $F_1$ and the other is a torus parallel to the boundary torus $\partial E(k'\# k)$.  Thus $Q$ is a Heegaard surface of the submanifold $E(k')\cup M_0$ of $E(k'\# k)$.

The amalgamation (along $F_1$ and $F_2$) of this Heegaard splitting of $E(k')\cup M_0$ and the minimal genus Heegaard splittings of $M_1$ and $M_2$ is a Heegaard splitting of $E(k'\# k)$.  The genus of this amalgamated Heegaard splitting is $g(M_1)+g(Q)+g(M_2)-g(F_1)-g(F_2)=g(M_1)+2n+g(M_2)-n-n=g(M_1)+g(M_2)$.  This means $g(E(k'\# k))\le g(M_1)+g(M_2)$.  By Lemma~\ref{Lg}, $g(E(k))=g(M_1)+g(M_2)$.  Thus $g(E(k'\# k))\le g(E(k))$.
\end{proof}



\end{document}